%% file: main.tex
\documentclass[a4paper]{article}

\input{settings}

\title{Applications of fast triangulation simplification}
\author{Mark C. Bell\footnote{Department of Mathematics, University of Illinois: \texttt{mcbell@illinois.edu}} \and Richard C. H. Webb\footnote{DPMMS, Centre for Mathematical Sciences, University of Cambridge: \texttt{rchw2@cam.ac.uk}}}

\begin{document}

\maketitle

\begin{abstract}
We describe a new algorithm to compute the geometric intersection number between two curves, given as edge vectors on an ideal triangulation.
Most importantly, this algorithm runs in polynomial time in the bit-size of the two edge vectors.

In its simplest instances, this algorithm works by finding the minimal position of the two curves.
We achieve this by phrasing the problem as a collection of linear programming problems.
We describe how to reduce the more general case down to one of these simplest instances in polynomial time.
This reduction relies on an algorithm by the first author to quickly switch to a new triangulation in which an edge vector is significantly smaller.
\end{abstract}

\keywords{triangulations of surfaces, geometric intersection number, flip graphs, Dehn twists, regular neighbourhood}

\ccode{57M20}  % 05C12

\section{Introduction}

Let $S$ be an (orientable) punctured surface and let $\zeta = \zeta(S) \defeq -3 \chi(S)$.
We will assume that $S$ is sufficiently complex that $\zeta \geq 3$ and so $S$ can be decomposed into an (ideal) triangulation.
Any such triangulation of $S$ has exactly $\zeta$ edges.

We can use a triangulation to give a combinatorial description of a curve.
The (essential, simple closed) curve $\gamma$ on $S$ is uniquely determined by its \emph{edge vector}:
\[ \calT(\gamma) \defeq \left(\begin{array}{c} \intersection(\gamma, e_1) \\ \vdots \\ \intersection(\gamma, e_\zeta) \end{array} \right) \in \NN_0^\zeta \]
where $e_1, \ldots, e_\zeta$ are the (ordered) edges of $\calT$.

In this paper we describe a new algorithm for computing the geometric intersection number $\intersection(\alpha, \beta)$ from $\calT(\alpha)$ and $\calT(\beta)$.
Most importantly, this algorithm runs in polynomial time in the bit-size of $\calT(\alpha)$ and $\calT(\beta)$.

To achieve this we focus on the simplest case when $\calT$ is \emph{$\alpha$--minimal}, that is, when $\calT$ minimises $\intersection(\alpha, \calT)$.
On such a triangulation the combinatorics of $\alpha$ are extremely restricted and so there are very few possibilities that we need to consider.
This allows us to reduce finding the minimal position for $\alpha$ and $\beta$ down to a collection of linear programming problems.
We show that these problems are sufficiently small that we can solve them, and so deduce $\intersection(\alpha, \beta)$, in polynomial time.

In the more general case, we apply a series of moves to convert the problem back to one on an $\alpha$--minimal triangulation.
These moves consist of flipping edges of the triangulation and performing powers of a Dehn twist along a short curve.
In \cite{BellSimplifying} the first author showed that there is always such a move which reduces $\intersection(\alpha, \calT)$ by a definite fraction.
We use this to bound the number of moves needed to reach an $\alpha$--minimal triangulation.

There are several other simplification results in other models of curves on surfaces \cite[Section~4]{AHT} \cite{DynnikovBraids} \cite{EricksonNayyeri} \cite{SchaeferSedgwick}.
However, in all of these other models it is very difficult to keep track of how another curve changes during the simplification process.
This makes it extremely hard to reduce the generic problem down to the $\alpha$--minimal case as we are required to track $\beta$ through these moves too.

\section{Minimal triangulations}
\label{sec:minimal}

In this section we consider the problem of putting $\alpha$ in minimal position with respect to $\beta$ when both of these are curves given on an $\alpha$--minimal triangulation $\calT$.

One case that is particularly straightforward is if $\alpha$ is \emph{non-isolating}, that is, if every component of $S - \alpha$ contains a puncture.
Here there is only one possibility for how $\alpha$ can appear on $\calT$.

\begin{lemma}[{\cite[Section~2.4.2]{BellThesis}}]
If $\alpha$ is non-isolating then $\intersection(\alpha, \calT) = 2$ and so $\alpha$ must appear on $\calT$ as shown in Figure~\ref{fig:non_isolating}. \qed
\end{lemma}

\begin{figure}[ht]
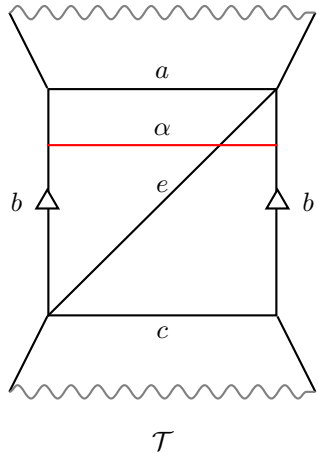

	\centering
	\include{tikz_non_isolating}
	\caption{An $\alpha$--minimal triangulation $\calT$ when $\alpha$ is non-isolating.}
	\label{fig:non_isolating}
\end{figure}

Hence in this configuration it is straightforward to put $\alpha$ in minimal position with respect to $\beta$ and so we can directly compute their intersection number.

\begin{proposition}
If $\alpha$ is non-isolating then, following the notation of Figure~\ref{fig:non_isolating},
\[ \intersection(\alpha, \beta) = \max(\mathbf{e} - \mathbf{b}, \mathbf{b} - \mathbf{e}, \mathbf{a} + \mathbf{c} - \mathbf{b} - \mathbf{e}) \]
where $\mathbf{x} \defeq \intersection(\beta, x)$. \qed
\end{proposition}

Thus we focus the remainder of this section on the case in which $\alpha$ is isolating.

\subsection{Combinatorial restrictions}

When $\alpha$ is isolating there are again many restrictions on its combinatorics.
The argument of \cite[Section~2.4.2]{BellThesis} shows that $\intersection(\alpha, e) \in \{0, 2\}$ for every edge $e$ of $\calT$.
This means that in each triangle $\alpha$ must appear either as a \emph{tripod} or \emph{corridor}, as shown in Figure~\ref{fig:tripod} and Figure~\ref{fig:corridor} respectively.

\begin{figure}[ht]
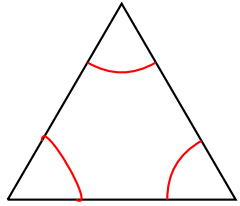
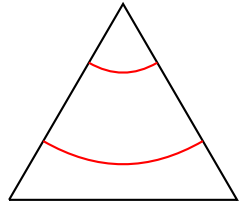

	\centering
	\begin{subfigure}[b]{0.45\textwidth}
		\centering
		\include{tikz_tripod}
		\caption{A tripod in a triangle.}
		\label{fig:tripod}
	\end{subfigure}
	~
	\begin{subfigure}[b]{0.45\textwidth}
		\centering
		\include{tikz_corridor}
		\caption{A corridor in a triangle.}
		\label{fig:corridor}
	\end{subfigure}
	\caption{An isolating curve $\alpha$ in an $\alpha$--minimal triangulation.}
\end{figure}

\begin{proposition}
\label{prop:tripod_corridor}
Let $g'$ be the genus of the component of $S - \alpha$ which does not contain any punctures.
Then $\alpha$ appears as a corridor in exactly one triangle of $\calT$ and as a tripod in exactly $4g' - 2$ triangles of $\calT$.
\end{proposition}

\begin{proof}
Clearly $\alpha$ must appear as a tripod in exactly $4g' - 2$ triangles due to the Euler characteristic of the unpunctured component.
Furthermore, if $\alpha$ did not appear as a corridor in any triangle of $\calT$ then it would be peripheral.
Hence it only remains to show that $\alpha$ appears as a corridor in at most one triangle.

Suppose instead that there are two triangles in which $\alpha$ appears as a corridor.
Let $e$ be the arc shown in Figure~\ref{fig:two_corridors}, which follows around $\alpha$ from the outside of one corridor to the other before connecting to a puncture.

\begin{figure}[ht]
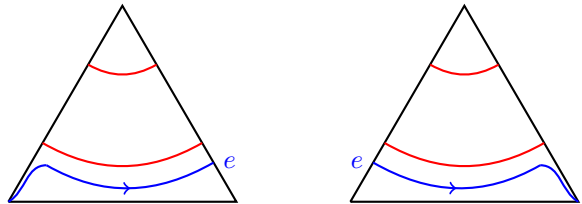

	\centering
	\include{tikz_two_corridors}
	\caption{A disjoint arc $e$ when $\alpha$ has two corridors.}
	\label{fig:two_corridors}
\end{figure}

We follow along $e$, flipping each edge of $\calT$ that we meet along the way.
Each flip reduces $\intersection(e, \calT)$ but does not increase $\intersection(\alpha, \calT)$ \cite[Page~38]{MosherFoliations}.
Thus, after performing at most $2 \zeta$ flips, we finish with a triangulation $\calT'$ which contains $e$ as an edge.
However, since $e$ is disjoint from $\alpha$ we must have that $\intersection(\alpha, \calT') < \intersection(\alpha, \calT)$.
This contradicts the fact that $\calT$ was $\alpha$--minimal.
\end{proof}

\subsection{Linear programming}

To finish the case in which $\alpha$ is isolating we formulate a collection of integer linear programming problems.
An optimal solution over all of these problems will then correspond to a minimal position.
The number of problems and, thanks to the fact that these curves are given on an $\alpha$--minimal triangulation, the number of variables involved will be bounded only in terms of $\zeta$.

To do this we first assign an orientation to each edge of $\calT$.
Additionally, for ease of notation throughout this section let $\mathbf{e_i} \defeq \intersection(\beta, e_i)$.

Fix $\mathfrak{b} \in \beta$ to be a representative which meets $\calT$ minimally.
Without loss of generality we may draw $\mathfrak{b}$ as a collection of straight line segments in each triangle.

Now for each edge $e_i$, choose $x_i, y_i \in \NN_0$ such that $x_i + y_i \leq \mathbf{e_i}$.
We construct a representative $\mathfrak{a} \in \alpha$ from these variables as follows.
\begin{enumerate}
\item If $\alpha$ meets the edge $e_i$ then we place two marks on the edge.
Walking along $e_i$ in the direction of its orientation, we place the first mark just after we encounter the $x_i\nth{}$ point of $\mathfrak{b}$.
Similarly, walking along $e_i$ in the reverse direction, we place the second mark just after we encounter the $y_i\nth{}$ point of $\mathfrak{b}$.
\item Each triangle now has exactly $0$, $4$ or $6$ marks on its boundary.
In the first case we do nothing in this triangle.
In the second and third cases we connect these via straight line segments to form a corridor or tripod respectively.
\item The union of these segments is our representative $\mathfrak{a} \in \alpha$.
\end{enumerate}
For example, see Figure~\ref{fig:tripod_constraints} where $x_i = 5$, $y_i = 3$, $x_j = 4$, $y_j = 3$, $x_k = 3$ and $y_k = 3$.
Alternatively, in the example shown in Figure~\ref{fig:corridor_constraints} we have that $x_i = 0$, $y_i = 0$, $x_j = 4$, $y_j = 3$, $x_k = 2$ and $y_k = 2$ but $\intersection(\alpha, e_i) = 0$.

\begin{figure}[ht]
	\centering
	\begin{subfigure}[b]{0.45\textwidth}
		\centering
		\input{tikz_LP_tripod}
		\caption{When $\alpha$ forms a tripod.}
		\label{fig:tripod_constraints}
	\end{subfigure}
	~
	\begin{subfigure}[b]{0.45\textwidth}
		\centering
		\input{tikz_LP_corridor}
		\caption{When $\alpha$ forms a corridor.}
		\label{fig:corridor_constraints}
	\end{subfigure}
	\caption{Positioning $\mathfrak{a}$ over $\mathfrak{b}$.}
	\label{rfidtag_testing}
\end{figure}
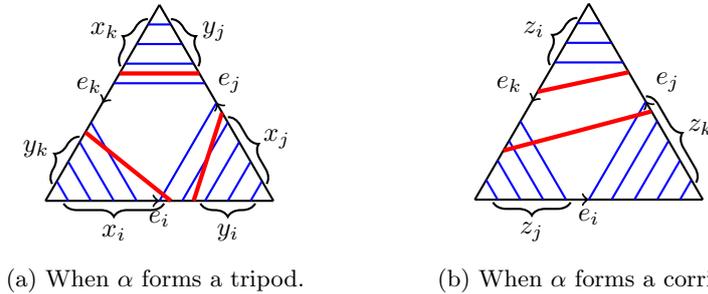

\begin{proposition}
\label{prop:intersection_PL}
The intersection number $\intersection(\mathfrak{a}, \mathfrak{b})$ is a piecewise linear function of $x_1, y_1, \ldots, x_\zeta, y_\zeta$.
\end{proposition}

\begin{proof}
Consider a single triangle of $\calT$ with sides $e_i$, $e_j$ and $e_k$.
Let $\mathbf{z_i}$, $\mathbf{z_j}$ and $\mathbf{z_k}$ denote the number of segments of $\mathfrak{b}$ running through this triangle parallel to the specified edge, as shown in Figure~\ref{fig:corridor_constraints}.
That is,
\[
\mathbf{z_i} \defeq \frac{1}{2}(\mathbf{e_j} + \mathbf{e_k} - \mathbf{e_i}), \;
\mathbf{z_j} \defeq \frac{1}{2}(\mathbf{e_i} + \mathbf{e_k} - \mathbf{e_j}) \; \textrm{and} \;
\mathbf{z_k} \defeq \frac{1}{2}(\mathbf{e_i} + \mathbf{e_j} - \mathbf{e_k}).
\]
Now consider a single segment $I$ of $\mathfrak{a}$ in this triangle which, without loss of generality, connects from $e_j$ to $e_k$.

If $I$ is part of a tripod, as shown in Figure~\ref{fig:tripod_constraints}, or is the segment of a corridor which is furthest from $e_i$, as shown in Figure~\ref{fig:corridor_constraints}, then:
\[ \intersection(I, \mathfrak{b}) = \begin{cases}
|x_k - y_j| & \textrm{if} \; x_k \leq \mathbf{z_i} \; \textrm{or} \; y_j \leq \mathbf{z_i} \\
x_k + y_j - 2\mathbf{z_i} & \textrm{otherwise}.
\end{cases} \]
Similarly, if $I$ is the segment of a corridor which is closest to $e_i$ then:
\[ \intersection(I, \mathfrak{b}) = \begin{cases}
|(\mathbf{e_j} - y_k) - (\mathbf{e_j} - x_j)| & \textrm{if} \; \mathbf{e_j} - y_k \leq \mathbf{z_i} \; \textrm{or} \; \mathbf{e_j} - x_j \leq \mathbf{z_i} \\
(\mathbf{e_j} - y_k) + (\mathbf{e_j} - x_j) - 2\mathbf{z_i} & \textrm{otherwise}.
\end{cases} \]
Note that both formulae are dependent on the orientations of $e_i$, $e_j$ and $e_k$ matching those in Figure~\ref{fig:tripod_constraints} and Figure~\ref{fig:corridor_constraints}.
In the event that the orientation on $e_i$ does not match, for example, the variables $x_i$ and $y_i$ must be interchanged.

In either case, this is a piecewise linear function of $x_1, y_1, \ldots, x_\zeta, y_\zeta$.
Hence by summing these functions over all segments in all triangles we see that $\intersection(\mathfrak{a}, \mathfrak{b})$ is a piecewise linear function of $x_1, y_1, \ldots, x_\zeta, y_\zeta$ too.
\end{proof}

We will denote this piecewise-linear function by $f$ and so
\[ \intersection(\mathfrak{a}, \mathfrak{b}) = f(x_1, y_1, \ldots, x_\zeta, y_\zeta). \]

Now any representative of $\alpha$ which is in minimal position with respect to $\mathfrak{b}$ is isotopic, relative to $\mathfrak{b}$, to some representative constructed by the above procedure.
Therefore there is a choice of $x_1, y_1, \ldots, x_\zeta, y_\zeta$ such that the corresponding $\mathfrak{a}$ is in minimal position with respect to $\mathfrak{b}$.
Thus the task of putting $\alpha$ in minimal position with respect to $\beta$ is equivalent to finding a minimum of $f$.

To find such a minimum we consider each piece of $f$ in turn.
We can formulate the problem of finding a minimum of $f$ on a piece as an integer linear programming problem.
There are many algorithms for solving such problems and, while they are $\mathbf{NP}$-complete in general \cite{GathenSieveking} \cite{GareyJohnson}, these can be solved in polynomial time as we have a fixed number of variables:

\begin{theorem}[{\cite{Eisenbrand} \cite{Lenstra}}]
\label{thrm:integer_LP}
Suppose that $m_0$ is fixed. There is an algorithm which, given a matrix $A \in \ZZ^{n \times m_0}$, vector $b \in \ZZ^{n \times 1}$ and vector $c \in \ZZ^{1 \times m_0}$ finds an optimal solution $x \in \ZZ^{m_0 \times 1}$ to the integer linear programming problem:
\begin{eqnarray*}
\textrm{Minimise} & c \cdot x \\
\textrm{Subject to} & A \cdot x \geq b.
\end{eqnarray*}
Moreover, this algorithm runs in polynomial time in the bit-size of $A$, $b$ and $c$. \qed
\end{theorem}

The bit-size of each of the above integer linear programming problems is at most $\log(\intersection(\beta, \calT))$.
Therefore by using the algorithm of Theorem~\ref{thrm:integer_LP} we can solve each of these problems in at most $O(\poly(\log(\intersection(\beta, \calT))))$ operations.

Finally, note that $\intersection(I, \mathfrak{b})$ from the proof of Proposition~\ref{prop:intersection_PL} is a piecewise linear function with $5$ pieces.
Therefore the number of intersections occurring in a triangle is a piecewise linear function with at most $5^3$ pieces.
Thus $f$ has at most $5^{2 \zeta}$ pieces and so there are at most $5^{2\zeta} \in O(1)$ such problems we must consider.
Hence we can also find the minimal solution over all problems, and so a minimum of $f$, in polynomial time in the bit-size of $\calT(\beta)$.

\begin{corollary}
Suppose we are given $\calT(\alpha)$ and $\calT(\beta)$ where $\calT$ is an $\alpha$--minimal triangulation.
Then we can compute the minimal position for $\alpha$ relative to $\beta$, and so $\intersection(\alpha, \beta)$, in polynomial time in the bit-size of $\calT(\beta)$. \qed
\end{corollary}

\begin{remark}
We may also view $\beta$ as a measured lamination.
Then minimal position occurs when there are no bigons between $\alpha$ and the underlying lamination of $\beta$.
Thus we can find the minimal position of $\alpha$ with respect to $\beta$ by solving a collection of linear programming problems instead of integer linear programming problems.
\end{remark}

\section{Flips and twists}
\label{sec:flips_twists}

We now consider the more general case, in which $\alpha$ and $\beta$ are given on a triangulation $\calT$ which is not $\alpha$--minimal.
To deal with this case, we introduce two basic moves for modifying triangulations; the \emph{flip} and the \emph{twist}.
We use these moves to give a polynomial time reduction back to the case in Section~\ref{sec:minimal}.

Firstly, we say that an edge of $\calT$ is \emph{flippable} if it is contained in two distinct triangles. If $e$ is such an edge then we may flip it to obtain a new triangulation $\calT'$ as shown in Figure~\ref{fig:flip}.

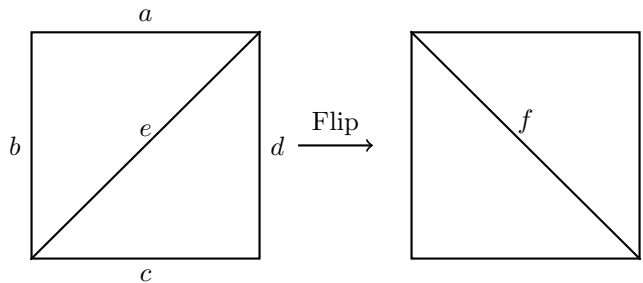
\begin{figure}[ht]
\centering
\input{tikz_flip}
\caption{Flipping an edge of a triangulation.}
\label{fig:flip}
\end{figure}

Secondly, if $\delta$ is a curve on $S$ then we may modify $\calT$ by performing the \emph{Dehn twist} $T_\delta^k$ \cite[Chapter~3]{FM}.
This move cuts the surface open along the curve $\delta$ and rotates one of the boundary components $k$ times to the right (or $|k|$ times to the left if $k$ is negative) before regluing the boundary components together.

In both cases it is straightforward to compute the edge vectors of $\alpha$ and $\beta$ on the new triangulation after performing such a move:

\begin{proposition}[{\cite[Page~30]{MosherFoliations}}]
\label{prop:flip_intersection}
Suppose that $\gamma$ is a curve and $e$ is a flippable edge of a triangulation $\calT$ as shown in Figure~\ref{fig:flip} then
\[ \intersection(\gamma, f) = \max(\intersection(\gamma, a) + \intersection(\gamma, c), \intersection(\gamma, b) + \intersection(\gamma, d)) - \intersection(\gamma, e). \]
Hence we can compute $\calT'(\gamma)$ in at most $O(\log(\intersection(\gamma, \calT)))$ operations. \qed
\end{proposition}

\begin{proposition}[{\cite{Schaefer07computingdehn}}]
\label{prop:twist_intersection}
Suppose that $\delta$ and $\gamma$ are curves. Given $\calT(\delta)$, $\calT(\gamma)$ and $k \in \ZZ$ we can compute $\calT'(\gamma)$ where $\calT' \defeq T_\delta^k(\calT)$ in at most
\[ O(\poly(\log(\intersection(\gamma, \calT)) + \log(\intersection(\delta, \calT)) + \log(k))) \]
operations. \qed
\end{proposition}

The usefulness of these moves comes from the fact that if a curve appears very complicated on $\calT$ then there is always such a move which reduces the number of intersections by a definite fraction:

\begin{theorem}[{\cite[Theorem~3.7]{BellSimplifying}}]
\label{thrm:fraction_intersection}
Let $D \defeq 80 \zeta B (10B + 1)^C$ where $B \defeq 5^{2\zeta}$ and $C \defeq 2^{2\zeta}$. If $\intersection(\gamma, \calT) > D$ then there is a triangulation $\calT'$ such that either:
\begin{itemize}
\item $\calT$ and $\calT'$ differ by a flip, or
\item $\calT' = T_\delta^k(\calT)$ where $|k| \leq \intersection(\gamma, \calT)$ and $\intersection(\delta, \calT) \leq 2 \zeta$
\end{itemize}
and $\intersection(\gamma, \calT') \leq (1 - 1/D) \intersection(\gamma, \calT)$. \qed
\end{theorem}

Thus, by using Theorem~\ref{thrm:fraction_intersection} at most $O(\log(\intersection(\alpha, \calT)))$ times we can obtain a triangulation $\calT'$ where $\intersection(\alpha, \calT') \leq D$.
To continue simplifying further we use the following lemma:

\begin{lemma}
\label{lem:drop_intersection}
If $\calT$ is not $\alpha$--minimal then by performing at most $2 \zeta$ flips we can reach a triangulation $\calT'$ such that $\intersection(\alpha, \calT') < \intersection(\alpha, \calT)$.
\end{lemma}

\begin{proof}
If $\intersection(\gamma, e) > 2$ for some edge $e$ of $\calT$ then there is an edge of $\calT$ which can be flipped in order to reduce the intersection number \cite[Lemma~2.4.3]{BellThesis}.
Hence we may assume that $\intersection(\alpha, e) \leq 2$ for each edge $e$ of $\calT$.

Now if $\intersection(\alpha, e) = 1$ for some edge $e$ then by performing at most two flips we can reduce the intersection number \cite[Lemma~2.4.4]{BellThesis}.

On the other hand, if $\intersection(\alpha, e) \in \{0, 2\}$ for every edge $e$ then $\alpha$ looks like a tripod or corridor in each triangle of $\calT$.
Hence, by the same argument as in the proof of Proposition~\ref{prop:tripod_corridor}, by performing at most $2 \zeta$ flips we can reach a triangulation with fewer intersections with $\alpha$.
\end{proof}

\begin{corollary}
\label{cor:general_minimal}
Given $\calT(\alpha)$ and $\calT(\beta)$ we can compute an $\alpha$--minimal triangulation $\calT'$ together with $\calT'(\alpha)$ and $\calT'(\beta)$ in at most
\[ O(\poly(\log(\intersection(\alpha, \calT)) + \log(\intersection(\beta, \calT)))) \]
operations. Furthermore, the bit-size of $\calT'(\beta)$ is at most
\[ O(\log(\intersection(\beta, \calT)) + \log^2(\intersection(\alpha, \calT))). \]
Thus we can compute minimal position representatives on $\calT'$, and so $\intersection(\alpha, \beta)$ in polynomial time in the bit-sizes of $\calT(\alpha)$ and $\calT(\beta)$ too.
\end{corollary}

\begin{proof}
By combining Theorem~\ref{thrm:fraction_intersection} and Lemma~\ref{lem:drop_intersection} we obtain a sequence of moves from $\calT$ to an $\alpha$--minimal triangulation $\calT'$.
By Proposition~\ref{prop:flip_intersection} and Proposition~\ref{prop:twist_intersection}, we can push $\calT(\alpha)$ and $\calT(\beta)$ through these moves and so obtain the edge vectors for these curves on $\calT'$ too.
As there are only $O(\log(\intersection(\alpha, \calT)))$ such moves, we can perform this computation in at most
\[ O(\poly(\log(\intersection(\alpha, \calT)) + \log(\intersection(\beta, \calT)))) \]
operations.

We now consider the effect of an individual move on $\intersection(\beta, \calT)$:
\begin{itemize}
\item Performing a flip increases $\intersection(\beta, \calT)$ by a factor of at most two by Proposition~\ref{prop:flip_intersection}.
\item Performing $T_\delta^k$ increases $\intersection(\beta, \calT)$ by at most $|k| \intersection(\delta, \calT) \intersection(\delta, \beta)$ \cite[Proposition~3.4]{FM}.
However for the twists that we will perform:
\begin{itemize}
\item $|k| \leq \intersection(\alpha, \calT)$,
\item $\intersection(\delta, \calT) \leq 2 \zeta$, and
\item $\intersection(\delta, \beta) \leq 2 \zeta \intersection(\beta, \calT)$.
\end{itemize}
Hence, performing this move increases $\intersection(\beta, \calT)$ by at most a factor of $4 \zeta^2 \intersection(\alpha, \calT)$.
\end{itemize}
Again, as only $O(\log(\intersection(\alpha, \calT)))$ such moves are performed, we have that
\[ \intersection(\beta, \calT') \in O\left(\intersection(\beta, \calT) \intersection(\alpha, \calT)^{\log(\intersection(\alpha, \calT))} \right) \]
and so the bound holds by taking logs.

We can now reapply the procedure of Section~\ref{sec:minimal} to compute minimal position for $\alpha$ and $\beta$ on $\calT'$ and so deduce $\intersection(\alpha, \beta)$.
As $\intersection(\alpha, \calT')$ and $\intersection(\beta, \calT')$ are sufficiently small, this can also be done in polynomial time.
\end{proof}

\section{Further extensions and applications}
\label{sec:extensions}

We finish with some further generalisations of the procedure of Corollary~\ref{cor:general_minimal}.

\subsection{Multicurves}
\label{sub:multicurves}

A slight variant of this procedure works even when $\alpha$ and $\beta$ have multiple components.

To handle this case we first use a polynomial time algorithm to extract the individual components and their multiplicities \cite[Section~4]{AHT} \cite[Section~4]{BellSimplifying} \cite[Section~6.4]{EricksonNayyeri}.
That is, we find $a_i, b_j \in \NN$ and curves $\alpha_i$, $\beta_j$ such that
\[ \alpha = \bigcup_i a_i \cdot \alpha_i \inlineand \beta = \bigcup_j b_j \cdot \beta_j. \]
We proceed by computing $\intersection(\alpha_i, \beta_j)$ for each $i$ and $j$ by the above procedure and then
\[ \intersection(\alpha, \beta) = \sum_{i,j} a_i b_j \intersection(\alpha_i, \beta_j). \]

\subsection{Multiarcs}
\label{sub:multiarcs}

The above procedure also works when $\alpha$ or $\beta$ is a \emph{multiarc}, that is, the isotopy class of the image of a smooth proper embedding of a finite number of copies of $[0, 1]$ (whose endpoints connect into punctures) into $S$.

If $\alpha$ is an arc then an $\alpha$--minimal triangulation is one which contains $\alpha$ as an edge.
Thus we perform the simplification routine to obtain $\alpha$ and $\beta$ on an $\alpha$--minimal triangulation $\calT'$. 
After this $\intersection(\alpha, \beta)$ is just the entry of $\calT'(\beta)$ associated to the edge $\alpha$.

If $\alpha$ is a multiarc then we extract its individual components and their multiplicities and proceed as in Section~\ref{sub:multicurves}.

The only modification needed to enable this is a slight change to how these multiarcs are represented combinatorially.
Since a non-trivial multiarc can have zero intersection with all edges, we make a slight modification to the standard definition of intersection number.

\begin{definition}
If $\alpha$ is a multiarc which contains $k$ copies of the edge $e$ of $\calT$ then their \emph{intersection number} is defined to be $\intersection(\alpha, e) \defeq -k$.
\end{definition}

This allows us to again represent a multiarc via its intersection numbers with the edges of a triangulation and for the procedure to work as before.

\subsection{Boundaries of neighbourhoods}
\label{sub:boundaries}

For curves $\alpha$ and $\beta$ let $\gamma \defeq \partial(N(\alpha \cup \beta))$.
When $\alpha$ and $\beta$ are in minimal position on an $\alpha$--minimal triangulation, it is straightforward to compute $\calT(\gamma)$ from $\calT(\alpha)$ and $\calT(\beta)$.
The same argument as in Section~\ref{sec:flips_twists} allows us to easily reduce this calculation on any triangulation back to one on an $\alpha$--minimal triangulation.
Thus, we can compute $\calT(\gamma)$ in polynomial time in the bit-sizes of $\calT(\alpha)$ and $\calT(\beta)$.

One key application for this is when $h$ is a reducible, aperiodic mapping class.
In this case, suppose that we have found a multicurve $\alpha$ such that $h^n(\alpha) = \alpha$.
We wish to upgrade from this $h^n$--invariant multicurve to an $h$--invariant one.
One way to achieve this is to take the $h$--invariant multicurve $\gamma \defeq \partial N(\alpha \cup h(\alpha) \cup \ldots \cup h^{n-1}(\alpha))$.
This multicurve is essential since $h$ is aperiodic and, by the above argument, we can compute $\calT(\gamma)$ in polynomial time in the size of $h$ and the bit-size of $\calT(\alpha)$.
Here we think of $h$ as being given as a sequence of edge flips starting from $\calT$ and its size is just the number of flips.

\begin{acknowledgements}
The first author acknowledges support from U.S. National Science Foundation grants DMS 1107452, 1107263, 1107367 ``RNMS: GEometric structures And Representation varieties'' (the GEAR Network).

The second author is supported by EPSRC Fellowship Reference EP/N019644/1. 
\end{acknowledgements}

\bibliographystyle{plain}
\bibliography{bibliography}

\end{document}

%% file: settings.tex
\usepackage{amssymb, amsmath, amscd, amsfonts, amsthm}
\usepackage{colonequals}
\usepackage[position=b]{subcaption}

\usepackage{microtype}
\usepackage{pinlabel}

\usepackage[
pdfborderstyle={},
pdfborder={0 0 0},
pagebackref,
pdftex]{hyperref}

\renewcommand*{\backref}[1]{}
\renewcommand*{\backrefalt}[4]{
  \ifcase #1 %
   [No citations.]%
  \or
   [#2]%
  \else
   [#2]%
  \fi
}

\newcommand{\calT}{\mathcal{T}}
\newcommand{\NN}{\mathbb{N}}
\newcommand{\ZZ}{\mathbb{Z}}

\newcommand{\defeq}{\colonequals}
\DeclareMathOperator{\intersection}{\iota}
\DeclareMathOperator{\poly}{poly}

\theoremstyle{plain}
\numberwithin{equation}{section}
\newtheorem{theorem}[equation]{Theorem}
\newtheorem{corollary}[equation]{Corollary}
\newtheorem{lemma}[equation]{Lemma}
\newtheorem{proposition}[equation]{Proposition}
\newtheorem{remark}[equation]{Remark}

\theoremstyle{definition}
\newtheorem{definition}[equation]{Definition}
\newtheorem*{definition*}{Definition}
\newtheorem*{keywords}{keywords}
\newtheorem*{acknowledgements}{Acknowledgements}

\newtheoremstyle{dotless}{}{}{}{}{\bfseries}{}{ }{}
\theoremstyle{dotless}
\newtheorem*{ccode}{\textbf{Mathematics Subject Classification (2010):}}

\newcommand{\fakeenv}{} %%% \fakeenv prints the emptystring

{ 
 \renewcommand{\fakeenv}{#2} %%% Now \fakeenv prints #2
 \theoremstyle{plain} 
 \newtheorem*{\fakeenv}{#1~\ref{#2}} %%% so now #2 is the name of a
 \begin{\fakeenv}
}
{
 \end{\fakeenv}
}

\newcommand{\inlineand}{\quad \textrm{and} \quad}

\newcommand{\nth}{\textsuperscript{th}}

\usepackage[outline]{contour}
\contourlength{0.1em}
\usepackage{color}
\usepackage{tikz}
\usetikzlibrary{calc}
\usetikzlibrary{decorations.pathreplacing}
\usetikzlibrary{decorations.pathmorphing}
\usetikzlibrary{decorations.markings}
\usetikzlibrary{shapes.geometric}

%% file: tikz_non_isolating.tex
\begin{tikzpicture}[scale=2,thick]

\tikzset{tri/.style={draw,scale=0.5,fill=white,regular polygon,regular polygon sides=3}}
\tikzset{snake/.style={decorate, decoration=snake}}

\node (rect) at (0, 0) [draw,minimum width=3cm,minimum height=3cm] {};

\draw (rect.south west) -- node [above] {$e$} (rect.north east);
\node [above] at (rect.north) {$a$};
\node [left] at ($(rect.west)+(-0.10,0)$) {$b$};
\node [below] at (rect.south) {$c$};
\node [right] at ($(rect.east)+(0.10,0)$) {$b$};

\draw [red] ($(rect.north west)!0.25!(rect.south west)$) -- node[black, above] {$\alpha$} ($(rect.north east)!0.25!(rect.south east)$);

\node [tri] (a) at (rect.west) {};
\node [tri] (a) at (rect.east) {};

\draw (rect.north west) -- ($(rect.north west) + (-0.25,  0.5)$);
\draw (rect.north east) -- ($(rect.north east) + ( 0.25,  0.5)$);
\draw (rect.south west) -- ($(rect.south west) + (-0.25, -0.5)$);
\draw (rect.south east) -- ($(rect.south east) + ( 0.25, -0.5)$);
\draw [snake, gray] ($(rect.north west) + (-0.25,  0.5)$) -- ($(rect.north east) + ( 0.25,  0.5)$);
\draw [snake, gray] ($(rect.south west) + (-0.25, -0.5)$) -- ($(rect.south east) + ( 0.25, -0.5)$);

\node (t) at ($(rect.south) + (0,-0.7)$) [anchor=north] {$\calT$};

\end{tikzpicture}

%% file: tikz_tripod.tex
\begin{tikzpicture}[scale=3,thick]

\coordinate (A) at (0,0);
\coordinate (B) at (1,0);
\coordinate (C) at ($(0.5,{sqrt(3) / 2})$);
\draw (A) -- (B) -- (C) -- (A);

\draw [red] ($(A)!0.3!(B)$) to [out=330, in=90] ($(A)!0.3!(C)$);
\draw [red] ($(B)!0.3!(C)$) to [out=210, in=90] ($(B)!0.3!(A)$);
\draw [red] ($(C)!0.3!(A)$) to [out=330, in=210] ($(C)!0.3!(B)$);

\end{tikzpicture}

%% file: tikz_corridor.tex
\begin{tikzpicture}[scale=3,thick]

\coordinate (A) at (0,0);
\coordinate (B) at (1,0);
\coordinate (C) at ($(0.5,{sqrt(3) / 2})$);
\draw (A) -- (B) -- (C) -- (A);

\draw [red] ($(B)!0.3!(C)$) to [out=210, in=330] ($(A)!0.3!(C)$);
\draw [red] ($(C)!0.3!(A)$) to [out=330, in=210] ($(C)!0.3!(B)$);

\end{tikzpicture}

%% file: tikz_two_corridors.tex
\begin{tikzpicture}[scale=3,thick]

\tikzset{midarrow/.style={
decoration={markings,mark=at position 0.5 with {\arrow{>}}},
postaction={decorate}
}}

\coordinate (A) at (0,0);
\coordinate (B) at (1,0);
\coordinate (C) at ($(0.5,{sqrt(3) / 2})$);
\draw (A) -- (B) -- (C) -- (A);

\draw [red] ($(B)!0.3!(C)$) to [out=210, in=330] ($(A)!0.3!(C)$);
\draw [red] ($(C)!0.3!(A)$) to [out=330, in=210] ($(C)!0.3!(B)$);

\coordinate (A2) at (1.5,0);
\coordinate (B2) at (2.5,0);
\coordinate (C2) at ($(2,{sqrt(3) / 2})$);
\draw (A2) -- (B2) -- (C2) -- (A2);

\draw [red] ($(B2)!0.3!(C2)$) to [out=210, in=330] ($(A2)!0.3!(C2)$);
\draw [red] ($(C2)!0.3!(A2)$) to [out=330, in=210] ($(C2)!0.3!(B2)$);

\draw [blue] (A) to [out=30, in=180] ($(A)!0.3!(C) + (280:0.1)$);
\draw [blue,midarrow] ($(A)!0.3!(C) + (280:0.1)$) to [out=330,in=210] ($(B)!0.2!(C)$) node [right] {$e$};

\draw [blue, midarrow] ($(A2)!0.2!(C2)$) node[left] {$e$} to [out=330,in=210] ($(B2)!0.3!(C2) + (260:0.1)$);
\draw [blue] ($(B2)!0.3!(C2) + (260:0.1)$) to [out=0,in=150] (B2);

\end{tikzpicture}

%% file: tikz_LP_tripod.tex
\begin{tikzpicture}[scale=3,thick]

\tikzset{midarrow/.style={
decoration={markings,mark=at position 0.5 with {\arrow{>}}},
postaction={decorate}
}}

\coordinate (A) at (0,0);
\coordinate (B) at (1,0);
\coordinate (C) at ($(0.5,{sqrt(3) / 2})$);
\draw [midarrow] (A) -- node [below] {$e_i$} (B);
\draw [midarrow] (B) -- node [above, xshift=5pt] {$e_j$} (C);
\draw [midarrow] (C) -- node [above, xshift=-5pt] {$e_k$} (A);

\draw [decorate,decoration={brace,amplitude=5pt,raise=1pt}] ($(A)!0.52!(B)$) -- ($(A)!0.08!(B)$) node [midway,yshift=-12pt] {$x_i$};
\draw [decorate,decoration={brace,amplitude=5pt,raise=1pt}] ($(A)!0.92!(B)$) -- ($(A)!0.68!(B)$) node [midway,yshift=-12pt] {$y_i$};

\draw [decorate,decoration={brace,amplitude=5pt,raise=1pt}] ($(B)!0.42!(C)$) -- ($(B)!0.08!(C)$) node [midway,xshift=12pt,yshift=5pt] {$x_j$};
\draw [decorate,decoration={brace,amplitude=5pt,raise=1pt}] ($(B)!0.92!(C)$) -- ($(B)!0.68!(C)$) node [midway,xshift=12pt,yshift=5pt] {$y_j$};

\draw [decorate,decoration={brace,amplitude=5pt,raise=1pt}] ($(C)!0.32!(A)$) -- ($(C)!0.08!(A)$) node [midway,xshift=-12pt,yshift=5pt] {$x_k$};
\draw [decorate,decoration={brace,amplitude=5pt,raise=1pt}] ($(C)!0.92!(A)$) -- ($(C)!0.68!(A)$) node [midway,xshift=-12pt,yshift=5pt] {$y_k$};

\foreach \i in {0.1,0.2,0.3,0.4} do \draw [blue] ($(A)!\i!(B)$) to ($(A)!\i!(C)$);
\foreach \i in {0.1,0.2,0.3,0.4,0.5} do \draw [blue] ($(B)!\i!(C)$) to ($(B)!\i!(A)$);
\foreach \i in {0.1,0.2,0.3,0.4} do \draw [blue] ($(C)!\i!(A)$) to ($(C)!\i!(B)$);

\draw  [red, ultra thick] ($(A)!0.55!(B)$) to ($(A)!0.35!(C)$);
\draw  [red, ultra thick] ($(B)!0.45!(C)$) to ($(B)!0.35!(A)$);
\draw  [red, ultra thick] ($(C)!0.35!(A)$) to ($(C)!0.35!(B)$);

\end{tikzpicture}

%% file: tikz_LP_corridor.tex
\begin{tikzpicture}[scale=3,thick]

\tikzset{midarrow/.style={
decoration={markings,mark=at position 0.5 with {\arrow{>}}},
postaction={decorate}
}}

\coordinate (A) at (0,0);
\coordinate (B) at (1,0);
\coordinate (C) at ($(0.5,{sqrt(3) / 2})$);
\draw [midarrow] (A) -- node [below] {$e_i$} (B);
\draw [midarrow] (B) -- node [above right] {$e_j$} (C);
\draw [midarrow] (C) -- node [above left] {$e_k$} (A);

\draw [decorate,decoration={brace,amplitude=5pt,raise=1pt}] ($(A)!0.42!(B)$) -- ($(A)!0.08!(B)$) node [midway,yshift=-12pt] {$z_j$};
\draw [decorate,decoration={brace,amplitude=5pt,raise=1pt}] ($(B)!0.52!(C)$) -- ($(B)!0.08!(C)$) node [midway,xshift=12pt,yshift=5pt] {$z_k$};
\draw [decorate,decoration={brace,amplitude=5pt,raise=1pt}] ($(C)!0.32!(A)$) -- ($(C)!0.08!(A)$) node [midway,xshift=-12pt,yshift=5pt] {$z_i$};

\foreach \i in {0.1,0.2,0.3,0.4} do \draw [blue] ($(A)!\i!(B)$) to ($(A)!\i!(C)$);
\foreach \i in {0.1,0.2,0.3,0.4,0.5} do \draw [blue] ($(B)!\i!(C)$) to ($(B)!\i!(A)$);
\foreach \i in {0.1,0.2,0.3} do \draw [blue] ($(C)!\i!(A)$) to ($(C)!\i!(B)$);

\draw  [red, ultra thick] ($(C)!0.75!(A)$) to ($(C)!0.55!(B)$);
\draw  [red, ultra thick] ($(C)!0.45!(A)$) to ($(C)!0.35!(B)$);

\end{tikzpicture}

%% file: tikz_flip.tex
\begin{tikzpicture}[scale=2,thick]

\node (rect) at (-1.5, 0) [draw,minimum width=3cm,minimum height=3cm] {};
\node (rect2) at (1, 0) [draw,minimum width=3cm,minimum height=3cm] {};

\draw (rect.south west) -- node [above] {$e$} (rect.north east);
\draw (rect2.north west) -- node [above, yshift=1] {$f$} (rect2.south east);

\node (a) at (rect.north) [anchor=south] {$a$};
\node (b) at (rect.west) [anchor=east] {$b$};
\node (c) at (rect.south) [anchor=north] {$c$};
\node (d) at (rect.east) [anchor=west] {$d$};

\draw [thick,->] ($(rect.east)!0.25!(rect2.west)$) -- node[above] {Flip} ($(rect.east)!0.75!(rect2.west)$);
\end{tikzpicture}

%% file: main.bbl
\begin{thebibliography}{10}

\bibitem{AHT}
Ian Agol, Joel Hass, and William Thurston.
\newblock The computational complexity of knot genus and spanning area.
\newblock {\em Trans. Amer. Math. Soc.}, 358(9):3821--3850, 2006.

\bibitem{BellSimplifying}
M.~C. {Bell}.
\newblock {Simplifying triangulations}.
\newblock {\em ArXiv e-prints}, April 2016.

\bibitem{BellThesis}
Mark Bell.
\newblock {\em Recognising mapping classes}.
\newblock PhD thesis, University of Warwick, 2015.

\bibitem{DynnikovBraids}
Ivan Dynnikov and Bert Wiest.
\newblock On the complexity of braids.
\newblock {\em J. Eur. Math. Soc. (JEMS)}, 9(4):801--840, 2007.

\bibitem{Eisenbrand}
Friedrich Eisenbrand.
\newblock Fast integer programming in fixed dimension.
\newblock In {\em Algorithms---{ESA} 2003}, volume 2832 of {\em Lecture Notes
  in Comput. Sci.}, pages 196--207. Springer, Berlin, 2003.

\bibitem{EricksonNayyeri}
Jeff Erickson and Amir Nayyeri.
\newblock Tracing compressed curves in triangulated surfaces.
\newblock {\em Discrete Comput. Geom.}, 49(4):823--863, 2013.

\bibitem{FM}
Benson Farb and Dan Margalit.
\newblock {\em A primer on mapping class groups}, volume~49 of {\em Princeton
  Mathematical Series}.
\newblock Princeton University Press, Princeton, NJ, 2012.

\bibitem{GareyJohnson}
Michael~R. Garey and David~S. Johnson.
\newblock {\em Computers and intractability}.
\newblock W. H. Freeman and Co., San Francisco, Calif., 1979.
\newblock A guide to the theory of NP-completeness, A Series of Books in the
  Mathematical Sciences.

\bibitem{Lenstra}
H.~W. Lenstra, Jr.
\newblock Integer programming with a fixed number of variables.
\newblock {\em Math. Oper. Res.}, 8(4):538--548, 1983.

\bibitem{MosherFoliations}
Lee Mosher.
\newblock Tiling the projective foliation space of a punctured surface.
\newblock {\em Trans. Amer. Math. Soc.}, 306(1):1--70, 1988.

\bibitem{SchaeferSedgwick}
Marcus Schaefer, Eric Sedgwick, and Daniel {\v{S}}tefankovi{\v{c}}.
\newblock Algorithms for normal curves and surfaces.
\newblock In {\em Computing and combinatorics}, volume 2387 of {\em Lecture
  Notes in Comput. Sci.}, pages 370--380. Springer, Berlin, 2002.

\bibitem{Schaefer07computingdehn}
Marcus Schaefer, Eric Sedgwick, and Daniel Stefankovic.
\newblock Computing {D}ehn twists and geometric intersection numbers in
  polynomial time, 2007.

\bibitem{GathenSieveking}
Joachim von~zur Gathen and Malte Sieveking.
\newblock A bound on solutions of linear integer equalities and inequalities.
\newblock {\em Proc. Amer. Math. Soc.}, 72(1):155--158, 1978.

\end{thebibliography}
